\documentclass[a4paper,12pt]{article}

\usepackage{amsmath,amsfonts,amssymb,amstext,amsthm,amscd,mathrsfs,amsbsy}
\usepackage{enumitem}

\def\ZZ{\mathbb{Z}}
\def\QQ{\mathbb{Q}}
\def\RR{\mathbb{R}}
\def\Rb{\mathbb{R}_{+}}
\def\CC{\mathbb{C}}

\def\KK{\mathrm{K}}
\def\h{\mathrm{h}}
\def\Ok{\mathfrak{o}}

\def\R{\mathrm{R}}
\def\D{\mathrm{D}}
\def\w{\omega}
\def\Res{\kappa}

\def\fa{\mathfrak{a}}
\def\fb{\mathfrak{b}}

\def\fd{\mathfrak{d}}
\def\fp{\mathfrak{p}}
\def\fn{\mathfrak{n}}

\def\Zk{\zeta_{\KK}}
\def\Xk{\xi_{\KK}}

\def\phik{\varphi_{\KK}}

\def\muk{\mu_{\KK}}
\def\Nuk{\nu_{\KK}}
\def\pk{\phi_{\KK}}

\def\M{\mathcal{M}}

\newcommand{\T}{\mathrm{T}}

	\newcommand{\abs}[1]{\left\vert#1\right\vert}
	\newcommand{\norm}[1]{\left\Vert#1\right\Vert}
	\newcommand{\floor}[1]{\lfloor#1\rfloor}
   
   \newcommand{\NN}{\mathfrak{N}}

\newtheorem{theorem}{Theorem}[section]
\newtheorem*{theorem*}{Theorem}

\newtheorem{lemma}[theorem]{Lemma}
\newtheorem{proposition}[theorem]{Proposition}
\newtheorem{remark}[theorem]{Remark}

\title{Discrete  Measures  and  the Extended Riemann Hypothesis}
\author{Estala--Arias, Samuel \\
  \small Divisi\'on de Posgrado, Facultad de Ingenier\'ia, Universidad Aut\'onoma de  Quer\'etaro,\\
  \small Cerro de las Campanas S/N, Las Campanas, \\
  \small 76010 Santiago de Quer\'etaro, Qro., M\'exico \\
  \small e--mail: samuel.estala@uaq.mx
  }
\date{\today}

\begin{document}

\maketitle

\begin{abstract}  In this work we 
show that the  Riemann hypothesis for  the
  Dedekind zeta--function  $\Zk(s)$ of  an algebraic number field $\KK$ is equivalent to a problem of   the rate of convergence  of  certain discrete measures defined arithmetically on  the multiplicative group of  positive real numbers to the measure $\Zk(2)^{-1}\Res q dq $, where $\Res$ denotes the residue of $\Zk(s)$ at $s=1$ and $dq$   the  Lebesgue measure. 
\end{abstract}


\section{Introduction}

Let $\varphi(n)=n \prod_{p \mid n }( 1-\frac{1}{p})$ be the Euler's totient function, which counts the cardinality of the group of units of the ring $\ZZ/n\ZZ$, as well as
the number of positive  integers which are   lesser or equal to $n$, and relatively prime to $n$. 
The average of this arithmetic function is given by a well--known result due to Mertens (\cite{Mer}):
$$\Phi(x)=\sum_{n\leq x} \varphi(n)=  \frac{3}{\pi^2}x^2+E(x)= \frac{3}{\pi^2}x^2+O(x \log x).
$$ The logarithmic term of this result has  subsequently  been enhanced  by many authors (see e.g. \cite{NV} and the references therein).  On the other hand, the connection of $\Phi(x)$ and its asymptotic behavior    to the Riemann zeta--function has been inaugurated in work of  Franel and  Landau  on Farey sequences (\cite{Fr, EL}).

Let $\Rb= \{ q \in \RR \mid q >0 \}$ be  the multiplicative group of  positive real numbers. Recall that as a topological Abelian group, $\Rb$ has   a Haar measure $\frac{dq}{q},
$ where $dq$ is the Lebesgue measure on $\Rb$. The measure  $m=\zeta(2)^{-1}qdq=\frac{6}{\pi^2}q dq$, where $\zeta(s)$ denotes the Riemann zeta--function,  is absolutely continuous with respect to  Haar measure.  For each $q \in \Rb$, define the infinite measure, $m_q: C_c^{0}(\Rb)\rightarrow \CC$, by the formula  
$$
m_q(f) =\sum_{n=1}^{\infty} q \varphi(n)f(q^{\frac{1}{2}}n). 
$$ 
Let $\ell\geq 0$ be an integer or infinity.
We denote by $C_c^{\ell}(\Rb)$  the set of complex valued functions defined on $\Rb$ of class $C^\ell$ with compact support and  also by $C(\Rb)$  the set of continuous complex valued functions on $\Rb$.

We have the following   outstanding results  due to A. Verjovsky (cf.  \cite{V2}): from Merten's theorem,  for any $f \in C_c^{0}(\Rb)$, we have $$m_q(f) =m(f)+E_f(q)=m(f)+O(q^{1/2} \log q ),$$ as  $q \rightarrow 0$,  which implies that the measures $m_q$  converge vaguely to  $m$  as $q  \rightarrow 0$. In addition, 
if $f \in C_c^{2}(\Rb)$, then $E_f(q)=o(q^{\frac{1}{2}})$, and 
the  Riemann hypothesis holds if and only if for every $f \in C_c^{2}(\Rb)$  one has $E_f(q)=o(q^{3/4-\epsilon})$,  for all $0<\epsilon<1/4$ ($q \rightarrow 0$). 
Furthermore, if $f$ is the characteristic function of an interval then the exponent $1/2$ of $q$ in  the error term is optimal, i.e. for any $\alpha>\frac{1}{2}$,  the value of $q^{\alpha}E_f(q)$ is not bounded as $q$ goes to zero. Last but not least,  there exists  a continuous function $F$ of bounded support  such that  the exponent $1/2$ of $q$ in the error term $E_F$ is optimal (in the sense above) if and only if the Riemann hypothesis  is false in the strongest possible sense: there exist zeroes of Riemann's zeta--function  arbitrarly close to the critical line $\Re(s)=1$.

It is worth noticing that these outcomes  do not disproof the Riemann hypothesis because characteristic functions are   not even continuous. 
However, these results put in evidence the fact that the Riemann hypothesis is also a regularity
problem: the Riemann-Lebesgue theorem for any function given by the restriction of     $\zeta(s-1)/s\zeta(s)$  to vertical lines on the critical strip (which  are
not  necessarily in $L^1$)  implies that the Riemann hypothesis is false.

The measures $m_q$ and their connection to the Riemann hypothesis were discovered by Verjovsky in the beautiful article \cite{V} (see also \cite{CP} and \cite{NV}), as a consequence of studying geometrically the work of Don Zagier \cite{DZ} and P. Sarnak \cite{Sa}, which respectively relate the distribution of the long closed horocycles, in the classical modular orbifold and, respectively, its unit tangent bundle, to the Riemann hypothesis.

The purpose of this article is to investigate  Verjovsky's results   in  the case of a general algebraic  number field $\KK$ of degree $n=[\KK:\QQ]>1$. It may be noted  that in \cite{Sam}, the author set up a generalization of  Zagier's criterion for an algebraic number field $\KK$ and  show that the Riemann hypothesis for the Dedekind zeta--function $\Zk(s)$ is equivalent to   a problem of the rate of convergence of   certain generalized horocycle measures on the Hilbert modular orbifold of the field $\KK$ to its normalized Haar measure. In addition, the connection of the uniform  distribution of the long closed    horospheres on   Bianchi modular orbifolds with the measures studied here has already been the subject of many works (see the recent survey \cite{ParkPau} and the references therein).

Let us  state  our conclusions. Let $\KK$ be an algebraic number field  of degree  $n>1$,   over $\QQ$.  Denote  by $\Ok = \Ok_K$, the ring of integers of $\KK$. We usually denote by the letters $\fa,\fb,\ldots$ integral ideals of $\Ok$ and by $\fp$ a prime integral ideal.  Let  $\Zk(s)$  the Dedekind zeta--function of the algebraic number field $\KK$ and by $\Res=\Res_\KK$, the residue of  $\Zk(s)$ at $s=1$. The Extended Riemann hypothesis for $\Zk(s)$ states that all of its   non trivial zeros are in the line $\Re(s)=\frac{1}{2}$. 
The Extended Lindel\"of  hypothesis   states that $\Zk(\frac{1}{2}+it)=O(t^\epsilon)$ for  all $\epsilon >0$. The Extended Riemann hypothesis for $\Zk(s)$ implies the Extended Lindel\"of hypothesis for $\Zk(s)$.

For each integral ideal $\fn$ define the Euler's totient function  of the field $\KK$, $\varphi_{\KK}(\fn)$, as the cardinality of the group of invertible elements of $\Ok/\fn$, i.e.
$$\varphi_{\KK}(\fn)=\abs{(\Ok/\fn)^{*}}=\NN(\fn) \prod_{\fp \mid \fn }( 1-\frac{1}{\NN(\fp)}),$$
where
$$ \NN(\fn) :=\left [ \Ok: \fn \right ]=\left|\Ok/\fn\right|$$
is  the ideal norm of the integral ideal $\fn$.
Notice that $\varphi_{\KK}(\fn)$ counts the number of principal integral ideals  which are  relatively prime to $\fn$, and whose ideal norm are lesser or equal to the ideal norm of  $\fn$.  For each $q \in \Rb$, define the infinite measure, $m_q(f):C_c^0(\Rb) \rightarrow \CC$ by the formula  
\begin{equation}\label{measures}
m_q(f) =\sum_{\substack{ \fn \subset \Ok \\ }} q \phik(\fn)f (q^{\frac{1}{2}}\NN(\fn)), 
\end{equation} where the sum is taken over the set of all integral ideals of $\Ok$.
Additionally, let $m$  be the   measure given by 
$m=\frac{\Res}{\Zk(2)}q dq$.

Denote by $\floor{x}$ the floor function on $\RR$. The following results encompass our achievements.

\begin{theorem} \label{thm1} 

\begin{enumerate}[label=(\Alph*)]
\item For all $f \in C_c^0(\Rb)$ we have
$$m_q(f)=m(f)+O(q^{\frac{1}{2n}}\log q) , \qquad(q \rightarrow 0).$$ 
Moreover, if the  Lindel\"of hypothesis for the Dedekind zeta--function $\Zk(s)$ holds, then for all $f \in C_c^0(\Rb)$ we have
$$m_q(f)=m(f)+O(q^{1/4-\epsilon}\log q)  \qquad(q \rightarrow 0)$$
for all $0<\epsilon<1/4$. Furthermore, if the hypothesis on the generalized circle problem for $\KK$ holds, i.e. the number of integral ideals with norm less than $x$ is equal to $\Res x+O(x^{\frac{1}{2}-\frac{1}{2n}+\epsilon})$ (for all $0<\epsilon$), then
$$m_q(f)=m(f)+O(q^{\frac{1}{4}+\frac{1}{2n}-\epsilon}\log q),  \qquad(q \rightarrow 0),$$
for all $0<\epsilon<\frac{1}{4}+\frac{1}{4n}$.

\item  If  $f \in C_c^\ell(\Rb)$ with $ \floor{\frac{n}{2}}+2 \leq \ell \leq \infty $, then
$$m_q(f)=m(f)+o(q^{\frac{1}{2}}),  \qquad(q \rightarrow 0).$$

\item
The  Riemann hypothesis for the Dedekind zeta--function $\Zk(s)$ holds if and only if for every    function $f \in C_c^{\ell}(\Rb)$ with $ n+1 \leq \ell \leq \infty $, one has 
\begin{equation}\label{distribucionideal}
m_q(f)=m(f)+o(q^{3/4-\epsilon}),  \qquad (q \rightarrow 0 ) \notag
\end{equation}
for all $0<\epsilon<1/4$. Furthermore, if $\alpha \in (1/2, 3/4)$ is such that, for all functions $f\in C^{\ell}_c(\Rb^r)$ with $n+1 \leq \ell \leq \infty $, one has

\begin{equation}\label{optimo}
m_q(f)=m(f)+o(q^{\alpha-\epsilon}),  \qquad (q\rightarrow 0),
\end{equation}
for all $0<\epsilon<1-\alpha$,
then the Dedekind zeta--function has no zeroes in the half--plane $\Re(s)>2(1-\alpha)$. Contrariwise, if the Dedekind zeta--function has no zeroes in the half--plane $\Re(s)>2(1-\alpha)$, then,  for all functions  $f\in C^{\ell}_c(\Rb)$ with $n+1 \leq \ell \leq \infty $, equation (\ref{optimo}) holds for all $0 <\epsilon < 1-\alpha$.

\item If $f$ is the characteristic function of an interval, then
$$\limsup_{q \to 0} q^{-\alpha}\abs{ m_q(f)-m(f)}=\infty, \qquad ( \text{if } \alpha > 1/2). $$

\item Let the function $F \in C(\Rb)$, with support in $(0,1]$, be defined by
$$
F(t)=\begin{cases} (1-t)^{\floor{\frac{n}{2}}+1}& \text{for } \leq 1,\\
      0 & \text{for } t > 1.
     \end{cases}
$$
Then:

$$
m_q(F)=m(F)+o(q^{1/2})  \qquad (q\rightarrow 0),$$
and,
$$\limsup_{q \to 0} q^{-\alpha}\abs{ m_q(F)-m(F)}=\infty \qquad ( \text{for all } \alpha > 1/2) $$
if and only if the Riemann hypothesis for the Dedekind zeta--function is false in the strongest possible sense: there exist zeroes of Dedekind's zeta--function
arbitrarily close to the critical line $\Re(s)=1$.

\item Let the function $F \in C(\Rb)$, with support in $(0,1]$, be defined by
$$
F(t)=\begin{cases} (1-t)^n& \text{for } \leq 1,\\
      0 & \text{for } t > 1.
     \end{cases}
$$
Then:
$$
m_q(F)=m(F)+o(q^{1/2}),  \qquad (q\rightarrow 0),$$
and,
$$\limsup_{q \to 0} q^{-\alpha}\abs{ m_q(F)-m(F)}=\infty$$
for all $ \alpha > 1/2+\theta$ with $0 \leq \theta < \frac{1}{4} $, if and only if the Riemann hypothesis for the Dedekind zeta--function is false in the sense that there  exist zeroes of Dedekind's zeta--function
arbitrary close to the  line $\Re(s)=1-2\theta$. 
\end{enumerate}

\end{theorem}

This manuscript is arranged as follows. In Section  \ref{preliminaries} the theory of the   Dedekind zeta--function is given. In Section \ref{Mellin}, the method of the Mellin transform of the measures $q^{-1}m_q $ and its analytical properties is presented. In Section \ref{Proofs} we state a generalization of Mertens' theorem and proof our statements. 

 \section{Preliminaries on algebraic number fields}\label{preliminaries}

 In this section, the theory of the relevant arithmetic functions over an algebraic number field is presented. We  quote the classical  books \cite{JN},  \cite{SL}, and \cite{Nar}, for a more comprehensive introduction to the arithmetic of algebraic number fields.

\subsection{The Dedekind zeta function.}

The  Dedekind zeta--function of the algebraic number field $\KK$ is 
defined by the Dirichlet series:
$$\Zk(s)= \sum_{\substack{ \fa \subset \Ok }} \frac{1}{\NN(\fa)^s}, \qquad (\Re(s)> 1),
$$ where  the sum is taken over all integral ideals of $\Ok$. The  series which defines $\Zk(s)$ is absolutely convergent for $\Re(s)> 1$, and uniformly convergent in $\Re(s)>1+\epsilon$, for any positive $\epsilon$. Hence it is a holomorphic function on the half--plane $\Re(s)>0$. Additionally, we have the Euler product
$$\Zk(s)=\prod_{\fp}(1-\NN(\fp)^{-s})^{-1}, \qquad (\Re(s) > 1),
$$
where  the product is taken over all prime integral ideals of $\Ok$.
The completed Dedekind zeta-function is defined by $\Zk(s)=\Lambda(s)\Zk(s),$ where as in \cite{SL}, $\Lambda(s)=2^{-r_2 s} \D^{\frac{s}{2}} 
\pi^{-\frac{ns}{2}}\Gamma\left(\frac{s}{2} \right)^{r_1} \Gamma(s)^{r_2}.
$  
By a theorem of Hecke, $\Xk(s)$ has a meromorphic continuation to the whole complex plane with only two sample poles at $s=0,1$ and it satisfies the functional equation
$$
\Xk(s)=\Xk(1-s).
$$

\begin{remark}
 The class number formula asserts that the residue $\Res=\Res_K$, of the  Dedekind zeta function  $\Zk(s)$ at $s=1$ is given by 
\begin{equation*}
\Res=Res_{s=1}\Zk(s) = \frac{2^{r}(2\pi)^{r_2} \h \R}{\w\sqrt{\abs{\D}}},
\end{equation*} where $r_1, r_2, \w, \h,\R,$ and $\D$ are, respectively, the number of real places, the number of complex places, the number of complex root of unity, the class number, the discriminant, and  the regulator, of the algebraic number field $\KK$.
\end{remark}

\subsubsection{The Dedekind zeta function and its relation with the M\"obius function an Euler's totient function}

Let $\muk$ be the M\"obius function of $K$, which is  defined
by 
$$
\muk(\fa)=\begin{cases} 1 & \text{ if } \fa= \Ok,\\ 
(-1)^k & \text{ if the ideal  is the product of } k \text{ distinct  prime ideals,}\\
0 & \text{otherwise.}
\end{cases}
$$
 The M\"obius function is multiplicative and satisfies the properties:
 \begin{enumerate}
  \item 
$$
\sum _{\fd \mid \fa}\muk (\fd)= \begin{cases}1 &  \text{if } \fa=\Ok,\\
    0 & \text{if } \fa \neq \Ok.\end{cases} $$
\item 
$$\phik(\fn)=\sum_{\fd \mid \fn}  \muk(\fd) \NN(\frac{\fn}{\fd}).$$
 \end{enumerate}
Then, by the theory of Dirichlet series, we  have the following formulae:

\begin{enumerate}
 \item 
 $$
\frac{1}{\Zk(s)}=\sum_{\substack{ \fa \subset \Ok}}\frac{\mu(\fa)}{\NN(\fa) ^{s}},  \qquad \text{for } \Re(s)>1.
$$
\item
  $$
\frac{\Zk(s-1)}{\Zk(s)}= \sum_{ \substack{ \fa \subset \Ok} } \frac{\phik(\fa)}{\NN(\fa)^s}, \qquad \text{for } \Re(s)>1.
$$
\end{enumerate}

\begin{remark}
The function 
\begin{equation}\label{funcionControladora}
\pk(s):=\frac{\Zk(2s-1)}{\Zk(2s)} \notag
\end{equation} will be of greatest importance as it governs many properties of the Mellin transform of $q^{-1}m_q$.
\end{remark}

\subsubsection{The Phragm\'en--Lindel\"of function of $\Zk(s)$}\label{PhrLin}

Let us recall the following facts about the order of growth of $\Zk(s)$ along vertical lines. For each real number $\sigma$ we define a number $\Nuk(\sigma)$ as the lower bound of the numbers  $l\geq 0$ such that
\begin{equation*}
 \Zk(\sigma+it)=O(\abs{t}^{l}) \mbox{ as } \abs{t}\rightarrow \infty. \notag
\end{equation*}
Then $\Nuk(s)$ has the following properties (see e.g. \cite[p. 266]{SL}):
\begin{enumerate}
 \item    $\Nuk$ is continuous, non-increasing, and never negative.
\item $\Nuk$  is convex downwards in the sense that the curve $y=\Nuk(\sigma)$ 
 has no points above the chord joining any two of its points.
\item  $\Nuk(\sigma)=0$   if  $\sigma \geq 1$  and  $\Nuk(\sigma)=n(\frac{1}{2}-\sigma)$ if  $\sigma  \leq 0$.
\end{enumerate}
The function $\Nuk(\sigma)$ is sometimes called the  Phragm\'en--Lindel\"of function of $\Zk(s)$. The Extended Lindel\"of hypothesis  for $\Zk(s)$ states that  $\Nuk(\frac{1}{2})=0$, viz.,  for any $\epsilon>0$,
\begin{equation*}
\Zk(\frac{1}{2}+it)=O(t^{\epsilon})    \text{  as  }\abs{t} \rightarrow \infty.
\end{equation*}

A classical result, essentially due to  Littlewood (\cite{lit}),  asserts that  the Riemann hypothesis for $\Zk(s)$ implies that the  Lindel\"of hypothesis  for $\Zk(s)$ holds (cf. \cite[p. 337]{Tit} and \cite[p. 267]{SL}). Explicitly, under the Extended Riemann hypothesis,   for any $\epsilon>0$
\begin{equation}\label{Littlewood}
\Zk(s)=O(t^{\epsilon})  \text{ and } \Zk(s)^{-1}=O(t^{\epsilon}),
\end{equation} for   $\ s=\sigma+it, \sigma>\frac{1}{2}  \text{ as }\abs{t} \rightarrow \infty$.

\begin{remark} From the identity
$$
\frac{1}{\Zk(s)}=\sum_{\substack{ \fa \subset \Ok}}\frac{\mu(\fa)}{\NN(\fa) ^{s}}, \qquad (\Re(s) > 1),
$$ 
it follows that $\frac{1}{\Zk(s)}$ is uniformly bounded on  half--planes of the form $\Re(s) > 1+ \epsilon$, for any $\epsilon>0$. Moreover, from the Landau prime ideal theorem, $1/ \Zk(s)$ is a holomorphic function on  the half--plane $\Re(s) \geq 1$, and for any $\epsilon>0$, we have $$\frac{1}{\Zk(\sigma +it)}=O(t^{\epsilon}), \qquad (\sigma \geq1).$$ 
\end{remark}

\section{A Mellin transform method}\label{Mellin}
In this section we  define the Mellin transform of the measures $m_q$ and state its analytical properties (see \cite{V2, Sa, Sam} and the book \cite{KB} ). 

Let $\ell$ be a non negative integer or infinity.  For each $f\in C_{c}^{\ell}(\Rb)$   consider the \emph{Mellin transform} of $m_q(f)q^{-1}$:
\begin{equation}\label{trasformadadeMellin}
\M(f,s):=\int_{0}^{\infty}m_q(f)q^{s-1}\frac{dq}{q} \qquad( \Re(s)>1).
\end{equation}

\begin{proposition} \label{convergenciadeM} Given $f\in C_c^\ell(\Rb)$ with $ 0 \leq \ell \leq \infty$ the integral defining $\M(f,s)$ converges absolutely in the half--plane $\Re(s)>1$ and uniformly in  strips of the form $1 < \sigma_{0} \leq \Re(s) \leq  \sigma_1  < \infty $. Hence it defines a  holomorphic function in the half--plane $\Re(s)>1$.
\end{proposition} 
\begin{proof} 
For any $f\in C_c^{0}(\Rb)$ write $\norm{f}_{\infty}=\sup_{\substack{q\in \Rb}}\{\abs{f(q)}\}$. Then, since $m_{i}(f,q) \leq \norm{f}_{\infty}$, for $\Re(s)>1$ :
\begin{equation}
\abs{\M_i(f,s)}\leq \norm{f}_{\infty}\left(\frac{\T^{\sigma-1}}{\sigma-1}\right),  \notag
\end{equation} where $\sigma=\Re(s)$ and $T$ is a value depending on $f$ such that $f(q)=0$ for $q>T$. Therefore, we have absolute convergence in $\Re(s)>1$ and  uniform convergence in strips of the form $1 < \sigma_{0} \leq \Re(s) \leq  \sigma_1  < \infty $. 
\end{proof}

Combining the definitions of the measures $m_q(f)$ (Eq. (\ref{measures})) and the Mellin transform (Eq. (\ref{trasformadadeMellin})), we get (if $ \Re(s)>1$):
\begin{align}
\M(f,s)&=\int_{0}^{\infty}\left( \sum_{\substack{\fa \subset \Ok }} q \phik(\fa)f\big(q^\frac{1}{2} \NN( \fa)\big)  \right) q^{s-2}\frac{dq}{q} \notag \\ 
&=\sum_{\substack{\fa \subset \Ok }}  \phik(\fa )\int_{0}^{\infty} f\big(q^\frac{1}{2} \NN(\fa)\big )   q^{s-1}\frac{dq}{q} \notag \\
&= \sum_{\substack{\fa \subset \Ok }} 2\frac{\phik( \fa )}{\NN(\fa  )^{2s} }\int_{0}^{\infty} f(q)  q^{2s-1}dq, \notag 
\end{align}
where the last equality follows by changing the variable: $q'=q^{1/2} \NN(\fa)$. Then, for $f \in C_{c}^{0}(\Rb)$ and $\Re(s)>1$, we have
\begin{equation}\label{RS}
\M_f(s)=2  \frac{\Zk(2s-1)}{\Zk(2s)} \int_{0}^\infty f(q)q^{2s-1}dq. 
 \end{equation}

Since the integral in the last expression represents an holomorphic function on the whole complex plane for any continuous function $f$ with compact support, the Mellin transform $\M_f(s)$  has the same properties of  
$$2\pk(s)=2 \frac{\Zk(2s-1)}{\Zk(2s)}.$$

Explicitly, $\M_{f}(s)$ has a meromorphic continuation to the whole complex plane  that is regular for  $\Re(s)\geq \frac{1}{2}$ except, possibly, for a simple pole at $s=1$ with residue
\begin{equation}\label{resizetaRS}
\mathcal{R}es_{s=1} (\M_f(s))=\frac{\Res}{ \Zk(2)}\int_{0}^{\infty} f(q) q dq. \notag
\end{equation}  
Furthermore,  the Riemann hypothesis for the Dedekind zeta--function $\Zk(s)$ holds if and only if for all $f\in C_{c}^{0}(\Rb)$ the function $\M_f(s)$ is regular for $\Re(s)>1/4$ except, possibly, for a simple pole at $s=1$ with residue given as above.

\begin{remark}The modified function $$\M^{\star}(f,s)=\Lambda(2s-1)\Zk(2s) \M(f, s)=\Xk(2s-1)\int_{0}^\infty f(q)q^{2s-1}dq$$  has a holomorphic continuation to the whole complex plane except, possibly, for simple poles at $s=0,1$ and satisfies the functional equation $$\M^{\star}(f,s)=\M^{\star}(f,1-s).$$
\end{remark}


\begin{lemma}\label{cotas_principales}
If  $f \in C_c^{\ell}(\Rb)$ with $  0 \leq \ell < \infty$. Then, there exists $t_0>0$, independent of $f$, such that
$$
\abs{\M_f(\sigma+it)}\leq \frac{\beta_f\,t^{\frac{n}{2}+\epsilon} }{(1+\abs{t})^\ell} \quad \text{ for } t>t_0, \ \epsilon>0,
$$
and $1/2 \leq \sigma \leq 2$,
where  $\beta_f$ is a constant depending on  the first $l$  derivatives of $f$.
\end{lemma}

\begin{proof}

From Equation (\ref{RS}), we have

$$\M_f(s)=2 \Zk(2s)^{-1}\Zk(2s-1)\int_{\Rb}q^{2s-1}f(q)dq.$$
Then, integrating by parts we obtain:
\begin{align}
\M_f(s)&=   \frac{2\Zk(2s)^{-1}\Zk(2s-1)(-1)^{\ell}}{2s(2s+1)\cdots(2s+\ell-1)} \int_{\Rb}q^{2s+\ell-1}f^{(\ell)}(q)dq. \notag
\end{align}
 Let $\T>0$ be such that $ supp f \subset \{\, q \in \Rb \mid q \leq \T \,\}$. Thus,
\begin{align} 
\M_f(s)&=   \frac{2\Zk(2s)^{-1}\Zk(2s-1)(-1)^{\ell}}{2s(2s+1)\cdots(2s+\ell-1)} \int_{0}^{T}q^{2s+\ell-1}f^{(\ell)}(q)dq. \notag
\end{align}
Notice that  
$$\abs{ \int_{0}^{T}q^{2s+\ell-1}f^{(\ell)}(q)dq} \leq \norm{f^{(\ell)}}_\infty \frac{T^{2\sigma+\ell}}{2\sigma+\ell}, \qquad ( s=\sigma+it).
$$ Then there exists   a constant $\beta_f$ depending only on   the first $l$  derivatives of $f$ such that
we can bound the absolute value of  the Mellin transform $\M_f(s)$ over the  vertical band $1/2 \leq\Re(s)\leq 2$  by $\beta_f$ times the absolute value of   $$ \frac{2\Zk(2s)^{-1}\Zk(2s-1)(-1)^{\ell}}{2s(2s+1)\cdots(2s+\ell-1)}.$$ 
 
 On the other hand, form the properties of the Phragm\'en--Lindel\"of function  of the Dedekind zeta function, if $\epsilon>0$, 
$$
\Zk(2s-1)=O(t^{\frac{n}{2}+\epsilon}), \qquad(s=\sigma+it),
$$
uniformly on $1/2 \leq \Re(s) \leq 2$. In addition, if $ 1/2\leq\Re(s)$, $\Zk(2s)^{-1}=  O(1)$.
Therefore, for all $\epsilon >0$,  we have
\begin{equation}
 \pk(\sigma+it)=O(\abs{t}^{n/2+\epsilon}),  \notag
\end{equation} as $\abs{t} \rightarrow \infty$, 
which implies that $\M(\sigma+it)=O(\abs{t}^{n/2-\ell+\epsilon})$, uniformly in $1/2\leq \sigma \leq 2$. This proves our claim.

\end{proof}

\begin{lemma}\label{cotas_principales_HR}  If the Riemann hypothesis for the Dedekind zeta function $\Zk(s)$ holds and $ f \in C_c^{\ell}(\Rb)$ with  $ n+1 \leq \ell \leq \infty $, then for every $0< \epsilon<\frac{1}{4}$, there exists $t_0>0$ such that
$$
\abs{\M_f(\sigma+it)}\leq \frac{\beta_f(\epsilon)}{(1+\abs{t})^{1+\epsilon}}, \qquad \text{ for } t>t_0,
$$
for all  $\frac{1}{4}+\epsilon \leq \sigma\leq 2$. Here $\beta_f(\epsilon)$ is a constant depending  on $\epsilon$ and a finite number of derivatives of $f$. 
\end{lemma}
\begin{proof} First, we estimate $\pk(s)=\Zk(2s)^{-1} \Zk(2s-1)$ in the region $\frac{1}{4}+\epsilon\leq\Re(s) \leq 2$, under the assumption of the Riemann hypothesis for the Dedekind zeta function $\zeta_{\KK}(s)$.  Since this implies that   $\Zk(2s)^{-1}=O(t^{(2n-1)\epsilon})$ and we already have $\Zk(2s-1)=O(t^{n(1 -2\epsilon)})$,  both uniformly in $\frac{1}{4}+\epsilon \leq\Re(s) \leq 2$, it follows that $\pk(s)=O(t^{n(1-\epsilon)})$, uniformly in $\frac{1}{4}+\epsilon\leq\Re(s) \leq 2$. Now  we can integrate by parts as in  Lemma \ref{cotas_principales}  to see that 
$f \in C_c^{\ell}(\Rb)$ with  $ n+1 \leq \ell \leq \infty $ ensures that, 
$$\abs{\M_f(\sigma+it)} =O\big((1+\abs{t})^{-(1+\epsilon)} \big), \qquad(\abs{
t}\rightarrow \infty),$$ 
 uniformly in $\frac{1}{4}+\epsilon\leq\sigma \leq 2$. This proves the assertion of our lemma.
\end{proof}

\begin{remark} Denote by $C_c^{\ell}(\Rb)^{\star}$  the topological dual of $C_c^{\ell}(\Rb)$. The function $\M: \{ \Re (s) >1 \} \rightarrow C_c^{\ell}(\Rb)^{\star}$, given by
$$
s \mapsto  \int_{0}^{\infty}m_q(\cdot)q^{s-2}dq, \qquad (\Re (s) >1 ),
$$ defines a weakly holomorphic function. For every $s$ such that $ \Re (s) >1 $, $\M(s)$ defines
an infinite measure on $\Rb$. When $\ell = \infty$, $\M(s)$ defines a holomorphic function
whose values are distributions of finite order.  The analytic continuation of  $\M(s)$ to the whole complex plane  is  a weakly meromorphic function with values in
the  distribution space of $\Rb$ (see \cite{Sa, V2, NV}).
\end{remark}

\section{Uniform distribution of discrete measures}\label{Proofs}

In this section we prove our statements. We start by proving a generalization of Mertens' theorem. For a very general statement on visible lattice points   over an algebraic number fields  see  \cite{Sit}.

\begin{lemma}\label{normas} 
\begin{enumerate}
 \item For any ideal fractional ideal $\fd$ and  every $x>1$, we have 
$$ \sum_{\substack{ \NN(\fn) \leq x \\\fd \mid \fn }} \NN (\fn/\fd)= \frac{\Res}{2 } \left(\frac{x}{\NN(\fd)}\right)^2 + O\left(\big(\frac{x}{\NN(\fd)}\big)^{2-\frac{1}{n}}  \right).$$

\item If the Lindel\"of hypothesis holds, for any ideal fractional ideal $\fd$ and  every $x>1$, we have 
$$ \sum_{\substack{ \NN(\fn) \leq x \\\fd \mid \fn }} \NN (\fn/\fd)= \frac{\Res}{2 } \left(\frac{x}{\NN(\fd)}\right)^2 + O\left(\big(\frac{x}{\NN(\fd)}\big)^{\frac{3}{2}+\epsilon}  \right).$$ 

\item If $N(x) = \sum_{\substack{ \NN(\fn) \leq x \\ \fn \subset \Ok}}  1=\Res x + O(x^{\frac{1}{2}-\frac{1}{2n}+\epsilon} )$ for any positive $\epsilon$, then  for  $x > 1$ and $\epsilon>0$ we have
$$ \sum_{\substack{ \NN(\fn) \leq x \\\fd \mid \fn }} \NN (\fn/\fd)= \frac{\Res}{2 } \left(\frac{x}{\NN(\fd)}\right)^2 + O\left(\big(\frac{x}{\NN(\fd)}\big)^{\frac{3}{2}-\frac{1}{2n}+\epsilon}  \right).$$ 
\end{enumerate}

 \end{lemma}

\begin{proof} First, if  $N(x) = \abs{ \fn \subset \Ok \mid  \NN(\fn) \leq x\}}$,  it is known that (see e.g. \cite{MO}),
 $$N(x) = \sum_{\substack{ \NN(\fn) \leq x \\ \fn \subset \Ok}}  1=\Res x + O\big(x^{1-1/n} \big).$$
  Therefore, for  any integral ideal $\fd$, we have 
\begin{align}
\sum_{\substack{  \NN(\fn) \leq x \\ \fd \mid\fn}} \NN (\fn/\fd) &= 
\sum_{\substack{ \NN(\fn') \leq x/\NN(\fd) \\ \fn'\subset \Ok}} \NN (\fn') \notag \\
&\leq \int_1^{\frac{x}{\NN(\fd)}} N(t)dt \notag \\
&=\int_1^{\frac{x}{\NN(\fd)}} \Res t + O\big(t^{1-1/n} \big)dt \notag \\
&= \frac{\Res}{2} \left(\frac{x}{\NN(\fd)}\right)^2 + \int_1^{\frac{x}{\NN(\fd)}} O\big(t^{1-1/n} \big)dt \notag\\ 
&= \frac{\Res}{2} \left(\frac{x}{\NN(\fd)}\right)^2 + O\left(\big(\frac{x}{\NN(\fd)}\big)^{2-\frac{1}{n}}  \right). \notag
\end{align} 
Now, if the Lindel\"of hypothesis holds, it follows that (see \cite{Tak} and \cite{SL} pp.271), for any  positive $\epsilon$,
 $$N(x) = \sum_{\substack{ \NN(\fn) \leq x \\ \fn \subset \Ok}}  1=\Res x + O\big(x^{\frac{1}{2}+\epsilon}\big).$$ Hence
 \begin{align}
\sum_{\substack{  \NN(\fn) \leq x \\ \fd \mid\fn}} \NN (\fn/\fd) 
&= \frac{\Res}{2} \left(\frac{x}{\NN(\fd)}\right)^2 + \int_1^{\frac{x}{\NN(\fd)}} O\big(t^{\frac{1}{2}+\epsilon} \big)dt \notag\\ 
&= \frac{\Res}{2} \left(\frac{x}{\NN(\fd)}\right)^2 + O\left(\big(\frac{x}{\NN(\fd)}\big)^{\frac{3}{2}+\epsilon}  \right). \notag
\end{align} 
Finally, if $N(x) = \sum_{\substack{ \NN(\fn) \leq x \\ \fn \subset \Ok}}  1=\Res x + O(x^{\frac{1}{2}-\frac{1}{2n}+\epsilon} )$ for any positive $\epsilon$, then for  $x > 1$ and $\epsilon>0$ we have
$$ \sum_{\substack{ \NN(\fn) \leq x \\\fd \mid \fn }} \NN (\fn/\fd)= \frac{\Res}{2} \left(\frac{x}{\NN(\fd)}\right)^2 + O\left(\big(\frac{x}{\NN(\fd)}\big)^{\frac{3}{2}-\frac{1}{2n}+\epsilon}  \right).$$

\end{proof}

For  a real number $x \geq 1$, define
\begin{equation}
\Phi_{\KK}(x)=\sum_{\substack{\NN(\fa) \leq x}} \phik(\fa). 
\end{equation} A general version of Mertens' theorem is the content of the following proposition.

\begin{proposition}\label{Mertens} 
\begin{enumerate}
 \item
For  $x > 1$ we have
\begin{equation}
\Phi_{\KK}(x)=\frac{\Res }{2\Zk(2)} x^2+O(x^{2-1/n}\log x) . \notag
\end{equation}

\item

If the Extended Lindel\"of hypothesis is true, for  $x > 1$ and $\epsilon>0$ we have
\begin{equation}
\Phi_{\KK}(x)=\frac{\Res }{2\Zk(2)} x^2+O(x^{\frac{3}{2}+\epsilon}\log x) . \notag
\end{equation}

\item
If $N(x) = \sum_{\substack{ \NN(\fn) \leq x \\ \fn \subset \Ok}}  1=\Res x + O(x^{\frac{1}{2}-\frac{1}{2n}+\epsilon} \log x)$ for any positive $\epsilon$, then for  $x > 1$ and $\epsilon>0$ we have
\begin{equation}
\Phi_{\KK}(x)=\frac{\Res }{2\Zk(2)} x^2+O(x^{\frac{3}{2}-\frac{1}{2n}+\epsilon} \log x) . \notag
\end{equation}

\end{enumerate}

\end{proposition}

\begin{proof} First, from the relation of the Euler totient function  and the M\"obius function of $\KK$, 

\begin{align}
\Phi_{\KK}(x) &=\sum_{\substack{\NN(\fa) \leq x}} \sum_{\fd \mid \fa} \muk(\fd) \NN(\frac{\fa}{\fd})  \notag \\
 &=\sum_{\substack{ \fd \subset \Ok \\ \NN(\fd) \leq x} } \muk(\fd)\left( \sum_{\substack{ \NN(\fa) \leq x \\ \fd \mid \fa }} \frac{\NN(\fa)}{\NN(\fd)} \right).\notag 
\end{align}

Since the series  
$$
\frac{1}{\Zk(2)}=\sum_{\substack{ \fd \subset \Ok}}\frac{\mu(\fd)}{\NN(\fd) ^{2}}
$$ is convergent, from Lemma \ref{normas}, one gets

\begin{align}
\Phi_{\KK}(x) &=\sum_{\substack{\fd \subset \Ok \\ \NN(\fd) \leq x}} \mu(\fd)f(\fd) \notag \\
  &=\frac{\Res}{2}\sum_{\substack{\fd \subset \Ok \\ \NN(\fd) \leq x}} \frac{ \mu(\fd)}{ \NN (\fd)^2} x^2+O\left(x^{2-1/n} \sum_{\NN(\fd) \leq x} \frac{1}{\NN(\fd)^{2-1/n} } \right) \notag\\
  &=\frac{\Res}{2}\sum_{\fd \subset \Ok} \frac{ \mu(\fd)}{ \NN (\fd)^2} x^2+O\left(x^{2-1/n} \int_1^x \frac{t^{1-\frac{1}{n}}}{t^{2-1/n} } dt\right) \notag\\
  &= \frac{\Res}{2\Zk(2)} x^2+O(x^{2-1/n} \log x). \notag 
\end{align} 
Now, if the Extended Lindel\"of hypothesis is true, for  $x > 1$ and $\epsilon>0$ we have

\begin{align}
\Phi_{\KK}(x) 
  &=\frac{\Res}{2\Zk(2)} +O\left(x^{\frac{3}{2}+\epsilon} \int_1^x \frac{t^{\frac{1}{2}+\epsilon}}{t^{\frac{3}{2}+\epsilon} } dt\right) \notag\\
  &= \frac{\Res}{2\Zk(2)} x^2+O(x^{{\frac{3}{2}+\epsilon}} \log x ) \notag 
\end{align}
Finally, if $N(x) = \sum_{\substack{ \NN(\fn) \leq x \\ \fn \subset \Ok}}  1=\Res x + O(x^{\frac{1}{2}-\frac{1}{2n}+\epsilon} )$ for any positive $\epsilon$, then for  $x > 1$ and $\epsilon>0$ we have
\begin{align}
\Phi_{\KK}(x) 
  &=\frac{\Res}{2\Zk(2)} +O\left(x^{\frac{3}{2}-\frac{1}{2n}+\epsilon} \log x\right) \notag
\end{align}
 
\end{proof}

\subsection{Proof of Theorem A}

Let $f=\chi_{[a,b]}$ be the characteristic function of the interval $[a,b]$, where $0<a<b$. From the generalized Mertens'
theorem (Proposition \ref{Mertens}), we have
\begin{align}
m_q(f) &=\sum_{\substack{a q^{-\frac{1}{2}} \leq  \NN(\fn) \leq bq^{ - \frac{1}{2}}}} q \phik(\fn) \notag\\
&=\frac{\Res}{2  \Zk(2)}(b^2-a^2)+O(q^{\frac{1}{2n}} \log q ) \notag\\
&= \frac{\Res}{ \Zk(2)} \int_{\Rb}uf(u)du +O(q^{\frac{1}{2n}}  \log q )  \notag
\end{align}
The last two statements can be   proved similarly, using  Proposition \ref{Mertens}.

\subsection{Proof of Theorem B}

Let us show that   the critical exponent of $q$ in the error term is $\frac{1}{2}$  when considering characteristic functions of intervals.
This result follows immediately from the next observation originally discovered by A. Verjovsky  in \cite{V, V2}.

\begin{lemma}  For all $\alpha > 1$ we have
$$\limsup x^\alpha\abs{\frac{\Phi_{\KK}(x)}{x^2}- \frac{\Res}{2 \Zk(2)}} =\infty.$$ 
\end{lemma}

\begin{proof}
 In order that we might derive a contradiction suppose that for some $\alpha > 1$ the result of our lemma fails. Then there is a $c > 0$ and a function $b_{\alpha}(x)$, depending on $\alpha$, such that $\abs{b_{\alpha}(x)} < c$ for $0 < x$ and such that
 $$\frac{\Phi_{\KK}(x)}{x^2}- \frac{\Res}{2 \Zk(2)} =\frac{b_{\alpha}(x)}{x^\alpha}.$$ 
Notice that, for all $x > 0$ we have
 $$\frac{\Phi_{\KK}(x+1)}{(x+1)^2}=\frac{\Phi_{\KK}(x)x^2}{x^2(x+1)^2}+\frac{\phik(\floor{x+1})}{(x+1)^2}, $$
so that we are able to ponder to consecutive values of $b_{\alpha}$ with 
 $$
L(x)=b_{\alpha}(x)\frac{x^2}{(x+1)^2}-b_{\alpha}(x+1)\left(\frac{x }{x+1}\right)^\alpha.
$$ 
Clearly, this is  a bounded expression and
\begin{align}
L(x) &=x^{\alpha}\left(\frac{\Phi_{\KK}(x)}{(x+1)^2}-\frac{\Res x^2}{2\Zk(2) (x+1)^2}\right) \notag \\
&-x^{\alpha}(\frac{\Phi_{\KK}(x)}{(x+1)^2}+\frac{\phik(\floor{x+1})}{(x+1)^2}-\frac{\Res}{2\Zk(2)}) \notag \\
&=x^\alpha \left( \frac{\Res}{2 \Zk(2)}\frac{2x+1}{(x+1)^2}- \frac{\phik(\floor{x+1}) }{(x+1)^2} \right).
\notag
\end{align}
 However, if  $x + 1$ is restricted  to taking only   values of prime integral ideals, so that
\begin{align}
\phik(x + 1) & = \abs{ \{ \fn \subset \Ok \mid  \NN(\fn) \leq x  \}} \notag \\
&=\Res x+O(x^{1-\frac{1}{n}}),  \notag
 \end{align} then $L(x)$ is unbounded. This contradiction completes the
proof of our lemma.
 \end{proof}

 \subsection{Proof of Theorem C}

 The Mellin transform $\M_f(s)$ of $m_q(f)q^{-1}$  is holomorphic for $\Re(s)\geq \frac{1}{2} $ 
except, possibly, for a simple pole at $s=1$ with residue
$$
m(f)=\frac{\Res}{\Zk(2)} \int_{\Rb} q f(q)dq.
$$
From Proposition \ref{convergenciadeM}, the Mellin inversion formula applies, and we have
\begin{equation}\label{MI}
m_q(f)=\frac{1}{2\pi i}\int_{b-i\infty}^{b+i\infty}\mathcal{M}_f(s)q^{1-s}ds, 
\end{equation} for any real number $b>1$. If $f\in C_c^{\ell}(\Rb)$ with $\floor{\frac{n}{2}}+2 \leq \ell \leq \infty$, then by the estimates of Lemma \ref{cotas_principales} and the Cauchy residue theorem,  we  can shift the path of integration in equation (\ref{MI}) to the line $\sigma=\frac{1}{2}$ to get
$$
m_q(f)= m(f) +\frac{1}{2\pi}\int_{-\infty}^{\infty}\mathcal{M}_f(\frac{1}{2}+it)q^{\frac{1}{2}} q^{-it}dt.
$$
Now, because $\mathcal{M}_f(\frac{1}{2}+it)$ is integrable (w.r.t. $dt$), the Riemann--Lebesgue theorem implies
$$
\lim_{q\rightarrow 0}\abs{\int_{-\infty}^{\infty}\mathcal{M}_f(\frac{1}{2}+it)q^{it}dt}=0
$$ 
Therefore,
$$
m_q(f)=m(f)+o(q^{1/2}) \qquad (q \rightarrow 0).
$$

\subsection{Proof of Theorem D}

Suppose that for all  $f\in C_c^{\infty}(\Rb)$, we have the following bound:
$$m_{q}(f)=m(f)+O(q^{3/4-\epsilon}) \qquad(q\rightarrow 0)
$$ for all $0<\epsilon<1/4$
and write $m_q(f)=m(f)+E_f(q)$. Let   $\T$ be sufficiently large such that $m_q(f)=0$ for $q>\T$.  Then,
\begin{align}
\mathcal{M}_{f}(s)&=\int_{0}^{\T}m_{q}(f)q^{s-2}dq \notag\\
&=\int_{0}^{\T}(m(f)+E_f(q))q^{s-2}dq \notag \\
&=\frac{m(f)\T^{s-1}}{s-1}+~\int_{0}^{\T}E_f(q)q^{s-2}dq. \notag
\end{align} 
Since $a(q)=O(q^{\frac{3}{4}-\epsilon})$ as $q \rightarrow 0$, the last integral 
 converges absolutely and uniformly in the half-plane $\Re(s)>\frac{1}{4}+\epsilon$,  it defines a holomorphic function in that half-plane. Therefore, 
$\mathcal{M}_f(s)$ is a holomorphic function in the region $\Re(s)>\frac{1}{4}+\epsilon$ except, possibly, for a pole  at $s=1$ with residue $m(f)$. Thus, the Riemann hypothesis for  the  Dedekind  zeta--function $\Zk(s)$ is true.

Conversely,  suppose the Riemann hypothesis for the Dedekind  zeta function  holds. Then,  $\Zk(2s)$ does not vanish  for $\Re(s)>1/4$ and
$\M_f(s)$ is holomorphic for $\Re(s)>1/4$ except, possibly, for a simple pole at $s=1$ with residue $m(f)$. Moreover, from Lemma \ref{cotas_principales_HR},  the  integral of $\M_f(s)q^{1-s}$ exists over the boundary of the band $\frac{1}{4}+\epsilon \leq \sigma \leq 2$, for all $0<\epsilon< 1/4$. Hence the Mellin inversion formula and the Cauchy's residue theorem  implies
$$
m_q(f)=\mathcal{R}es_{s=1}(\M_f(s))+\frac{1}{2\pi}\int_{-\infty}^{\infty}\M_f(\frac{1}{4}+\epsilon +it)q^{-it}q^{\frac{3}{4}-\epsilon}dt.
$$ Again, by the Riemann--Lebesgue theorem:
$$
m_q(f)=m(f)+o(q^{\frac{3}{4}-\epsilon}),  \qquad (q\rightarrow 0).
$$

\subsection{Proof of Theorems E and F}

For  any integer $r \geq 1$, let  $F_r \in C(\Rb)$  be defined by
$$
F_r(t)=\begin{cases} (1-t)^r& \text{for } \leq 1,\\
      0 & \text{for } t > 1.
     \end{cases}
$$
Then the Mellin transform of $\M_F(s)$ is given by
$$
\M_{F_r}(s)=\frac{\Zk(2s-1)}{\Zk(2s)} \int_{0}^1(1-t)^r t^{2s-1}dt.
$$
Since the last integral is given by the Beta function 
$$B(r+1,2s)=\frac{\Gamma(r+1) \Gamma(2s)}{\Gamma(2s+r+1)},$$
it follows that $$
\M_{F_r}(s)=\frac{\Zk(2s-1)\Gamma(r+1)}{\Zk(2s) \prod_{k=0}^{r} (2s+k)} .
$$
Hence the only poles of $\M_F(s)$  in the half-plane $\Re(s) > 0$ are located at the zeroes of $\Zk(2s)$, since $2s(2s+1)\cdots(2s+r)$ does not vanish in that half--plane. 

Let $F=F_{\floor{\frac{n}{2}}+1}$ be the function given in Theorem (E). Notice that $\M_{F}(\sigma+it) = O(\frac{1}{(1+\abs{t})^{1+\frac{1}{4}}})$, uniformly in the vertical strip $ \frac{1}{2}\leq \sigma \leq 2 $, because $\pk(\sigma+it)=O(t^{\frac{n}{2}+\epsilon})$ uniformly in that vertical strip, for any  $\epsilon>0 $. Then we are able to shift the vertical
line of integration in Mellin's inversion formula to the vertical line $\Re(s) = \frac{1}{2}$, to apply the Riemann--Lebesgue theorem, and to obtain  the error term $E_F=o(q^\frac{1}{2})$. This  proofs the first assertions of theorems (E) and (F).

Now  suppose:
$$\limsup_{q \to 0} q^{-\alpha}\abs{ m_q(F)-m(F)}=\infty \qquad ( \text{for all } \alpha > 1/2) $$
Let $\beta$ be the supremum of the real parts of the zeroes of the Dedekind zeta--function and suppose that $\beta<1$. Here notice that, for sufficiently small $0<\epsilon<\frac{1}{4n}$,  $\pk(\sigma+it)=O( t^{n(\frac{1}{2}-\epsilon)+\epsilon'})$ uniformly in the vertical strip $ \frac{1}{2}-\epsilon \leq \sigma \leq 2 $, for any $\epsilon'>0$, which implies that  $\M_{F}(\sigma+it) = O(\frac{1}{(1+\abs{t})^{1+\frac{1}{4}}})$ uniformly in the vertical strip $ \frac{1}{2}-\epsilon \leq \sigma \leq 2 $, for sufficiently small $0<\epsilon<\frac{1}{4n}$.  Then we are able to shift the vertical
line of integration in Mellin's inversion formula to the vertical line $\Re(s) = \frac{1}{2} - \epsilon$,  to apply the Riemann--Lebesgue theorem, and to improved the  exponent $\alpha $ of the error term $E_F=o(q^\alpha)$  to be $\frac{1}{2}+\epsilon$, for some sufficiently small $0<\epsilon<\frac{1}{4n}$. This contradiction proofs the second assertion of Theorem (E).

In order to prove Theorem (F) consider    the function $F=F_{n}$ given in Theorem (E) and suppose:
$$\limsup_{q \to 0} q^{-\alpha}\abs{ m_q(F)-m(F)}=\infty,  $$
 for all  $\alpha > 1/2+\theta$ with $0 \leq \theta < \frac{1}{4}$.  Let $0 \leq \theta < \frac{1}{4}$ be fixed and suppose that  the supremum of the real parts of the zeroes of the Dedekind zeta--function satisfies $\beta<1-2\theta$. In this case, for sufficiently small $0<\epsilon<\frac{1}{4n}$, $\pk(\sigma+it)=O(t^{n})$ uniformly in the vertical strip $ \frac{1}{2}-\theta-\epsilon \leq \sigma \leq 2 $, which implies that  $\M_{F}(\sigma+it) = O(\frac{1}{(1+\abs{t})^{1+\frac{1}{2}}}),$ uniformly in the vertical strip $ \frac{1}{2}-\theta-\epsilon \leq \sigma \leq 2 $, for  sufficiently small $0<\epsilon<\frac{1}{4}-\theta$. Hence we are able to shift the vertical
line of integration in Mellin's inversion formula to the vertical line $\Re(s) = \frac{1}{2}-\theta - \epsilon$, for any $0<\epsilon<1/4-\theta$,   to apply the Riemann--Lebesgue theorem, and  to improved the exponent  $\alpha $ of $q$ in the error term $E_F=o(q^\alpha)$  to be $\frac{1}{2}+\theta +\epsilon$, for some sufficiently small $0<\epsilon<1/4-\theta$. This contradiction proofs the second assertion of Theorem (F).

\section*{Acknowledgments} 
The author would like to  thank A. Verjovsky and M. Cruz--L\'opez for  suggesting and  encouraging  the research of this remarkable subject.

\end{document}